\documentclass[11pt]{article}
\usepackage{latexsym}
\usepackage{amssymb,amsmath,amsthm}
\usepackage{addfont} 
\addfont{OT1}{rsfs10}{\rsfs}
\usepackage{graphicx}
\RequirePackage{tikz}
\usepackage{epsfig}
\usepackage{multicol}
\usepackage{pst-grad} 
\usepackage{pst-plot} 
\usepackage{pst-node} 
\usepackage{pst-tree}
\usetikzlibrary{shapes, arrows}
\usepackage[all]{xy}
\graphicspath{{fig/}}
\usepackage{subfig}
\usepackage{amsfonts}

\theoremstyle{definition}
\newtheorem{theorem}{Theorem}[section]
\newtheorem{lemma}[theorem]{Lemma}

\newtheorem{remark}[theorem]{Remark}
\newtheorem{corollary}[theorem]{Corollary}
\newtheorem{proposition}[theorem]{Proposition} 
\newtheorem{problem}[theorem]{Problem}
\newtheorem{definition}[theorem]{Definition}

\newcommand{\T}{\mathrm}

\begin{document}

\title{{\bf Uniquely Distinguishing Colorable Graphs}}

\author{M.~Korivand$^{1}$, N. Soltankhah$^{1,}$\thanks{Corresponding author} , and K.~Khashyarmanesh$^{2}$}

\date{}

\maketitle

\begin{center} 

$^{1}$ Department of Mathematics, Faculty of Mathematical Sciences, Alzahra University, Tehran, Iran
e-mail: {\tt m.korivand@alzahra.ac.ir or mekorivand@gmail.com}, {\tt soltan@alzahra.ac.ir}

$^{2}$ Department of Pure Mathematics, Ferdowsi University of Mashhad, \\
P.O.Box 1159-91775, Mashhad, Iran. \\
e-mail: {\tt khashyar@ipm.ir}

\end{center}

\begin{abstract}
A graph is called uniquely distinguishing colorable if there is only one partition of vertices of the graph that forms distinguishing coloring with the smallest possible colors. In this paper, we study the unique colorability of the distinguishing coloring of a graph and its applications in computing the distinguishing chromatic number of disconnected graphs. We introduce two families of uniquely distinguishing colorable graphs, namely type 1 and type 2, and show that every disconnected uniquely distinguishing colorable graph is the union of two isomorphic graphs of type 2. We obtain some results on bipartite uniquely distinguishing colorable graphs and show that  any uniquely distinguishing $n$-colorable tree with $ n \geq 3$ is a star graph. 
For a connected graph 
$G$, 
we prove that  
$\chi_D(G\cup G)=\chi_D(G)+1$ 
if and only if
$G$ 
is uniquely distinguishing colorable of type 1. 
Also, a characterization of all graphs 
$G$ of order $n$ with the property that $\chi_{D}(G\cup G) = \chi_{D}(G) = k$, where $k=n-2, n-1, n$, 
is given in this paper.
Moreover, we determine all graphs 
$G$ of order $n$ with the property that $\chi_{D}(G\cup G) = \chi_{D}(G)+1 = \ell$, where $\ell=n-1, n, n+1$. 
Finally, we investigate the family of connected graphs $G$ with $\chi_{D}(G\cup G) = \chi_{D}(G)+1 = 3$. 
\end{abstract}


\medskip\noindent
{\bf AMS Subj.\ Class}: 05C15. 

%%%%%%%%%%%%%%%%%%%%%%%%%%%%%%%%%%%%%%%%%%%%%%%%%%%%%%%%%%%%%%%%
%%%%%%%%%%%%%%%%%%%%%%%%%%%%%%%%%%%%%%%%%%%%%%%%%%%%%%%%%%%%%%%%
\section{Introduction}
\label{sec:intro}
%%%%%%%%%%%%%%%%%%%%%%%%%%%%%%%%%%%%%%%%%%%%%%%%%%%%%%%%%%%%%%%%
%%%%%%%%%%%%%%%%%%%%%%%%%%%%%%%%%%%%%%%%%%%%%%%%%%%%%%%%%%%%%%%% 
In 1977, a concept was introduced by Babaei, which became the basis of one of the most important methods for partitioning and distinguishing members of graphs 
by automorphism \cite{babai}. This concept, which we now know as {\it asymmetric coloring} (or {\it distinguishing labeling}), was added to the graph theory literature in 1996 by Albertson and Collins \cite{alber}, motivated by the following problem:\\

\textit{Professor X, who is blind, keeps keys on a circular key ring. Suppose there are a variety of handle shapes available that can be distinguished by touch. Assume that all keys are symmetrical so that a rotation of the key ring about an axis in its plane is undetectable from an examination of a single key. How many shapes does Professor X need to use in order to keep $n$ keys on the ring and still be able to select the proper key by feel?}\\

In the language of graph theory, this problem is
connected with the minimum number of colors needed to color the vertices of a
cycle graph such that there is a non-identity automorphism of the graph which preserves
all the vertices label. This topic has attracted a lot of attention. 
Google Scholar lists more than 340 citations to the papers of Albertson and Collins 1996 on ``Symmetry Breaking in Graphs". 
First, this concept was completely brought up in the field of group theory. 
So, the distinguishing number was computed for the automorphism group of graphs instead of the graphs, 
and the proof techniques were in the field of group theory.
But the research path changed quickly and the graph theory researchers took up the topic.

However, two research subfields were followed in this concept. Experts in group theory, 
by taking into account that the automorphism group of graphs acts on the vertices of that graph, 
studied this concept in the general case when a group acts on an arbitrary set.
For some outstanding achievements in this topic, we can refer to the articles \cite{Babai, Bailey, Chan, Klavžar, Tymoczko}.

Later, the distinguishing number was assigned to a graph and the proof techniques changed from the structure of automorphism group to the properties of automorphism in graphs such as distance preservation. 
This event had many advantages. In fact, we could color graphs that we do not even know their automorphism group. Because determining the automorphism group of a graph is a difficult problem and we are not interested in making the coloring problem dependent on it. 
Many papers have been written in such a way. For some interesting articles, we can refer to \cite{Ahmadi, Collins, Collins2, Tucker,  Kalinowski, Imrich, Klavžar, Russell, Shekarriz}.

In 2006, Collins and Trenk combined this concept with proper coloring and proposed a new coloring called {\it proper distinguishing coloring} (or {\it distinguishing coloring}) \cite{Collins}.  
This coloring has attracted the attention of many researchers and a large number of articles have been published about it. 
To see some recent results related to this coloring, the reader can refer to the articles \cite{Kalinowski2, Lehner, Meslem}.

In \cite{1}, Harary, Hedetniemi and Robinson introduced and studied the uniquely colorable graphs. In their  work  `coloring' means the  `proper coloring'. The interested reader may refer to the papers \cite{6} and \cite{7} for further  results about uniquely colorable graphs. 

Details of the definitions are described as follows. All graphs considered in this paper are simple. Let $G$ be a graph with vertex set $V(G)$. 
 A coloring (or labeling) of a graph $ G $ is a partition of the vertex set of $G$ into classes,
called the color classes. If a coloring contains exactly $n$ disjoint non-empty color classes, then it is called an $n$-coloring.
The symbol of distinguishing labeling  of  
 $ G $  will always be denoted by $ \left[  D(G) \right]$. 
We say that a coloring $[D(G)]$ with color classes $V_1, \ldots, V_{\ell}$ of $G$ is  distinguishing labeling if 
there is no non-trivial automorphism $f$ of $G$ with $f(V_i)=V_i$ for all $i=1, \ldots, \ell$. We denote the minimum  such $\ell$ by $D(G)$ and is called distinguish number of $G$.  A distinguishing labeling  $[D(G)]$ is distinguishing coloring (or proper distinguishing coloring) if it provides a proper coloring for $G$. 
The distinguishing chromatic number of a graph $G$,  denoted by $ \chi_{D}(G) $,  is the minimum $\ell$ such that
$[D(G)]=\{V_1, \ldots, V_{\ell} \}$ is a distinguishing coloring. 

Recall that a graph which has a proper $k$-coloring with no proper $(k - 1)$-coloring is $k$-chromatic. 
We say that a graph is uniquely distinguishing $ n $-colorable if it has exactly one distinguishing $n$-coloring.  
Furthermore,  we say  a graph is uniquely distinguishing colorable if there is only one partition of its vertex set
into the smallest possible number of distinguishing color classes.
  Actually a uniquely distinguishing colorable graph is a uniquely distinguishing $\chi_{D}(G)$-colorable. The symbols  of  proper coloring  and distinguishing coloring of  
 $ G $  will always denote $ \left[  \chi (G) \right]$  and  $ \left[  \chi_{D}(G) \right]$, respectively. We say vertices $ u $ and $ v $ are siblings if they have a common neighbor.
Finally, a fixed vertex of a graph $G$ is a vertex that is mapped to itself by any automorphism of $G$.  
We use \cite{3} for any terminology and notation not defined here.
%\newpage
\begin{figure}[!h]
\begin{center}
\label{FigE1}
  \includegraphics[width=10cm]{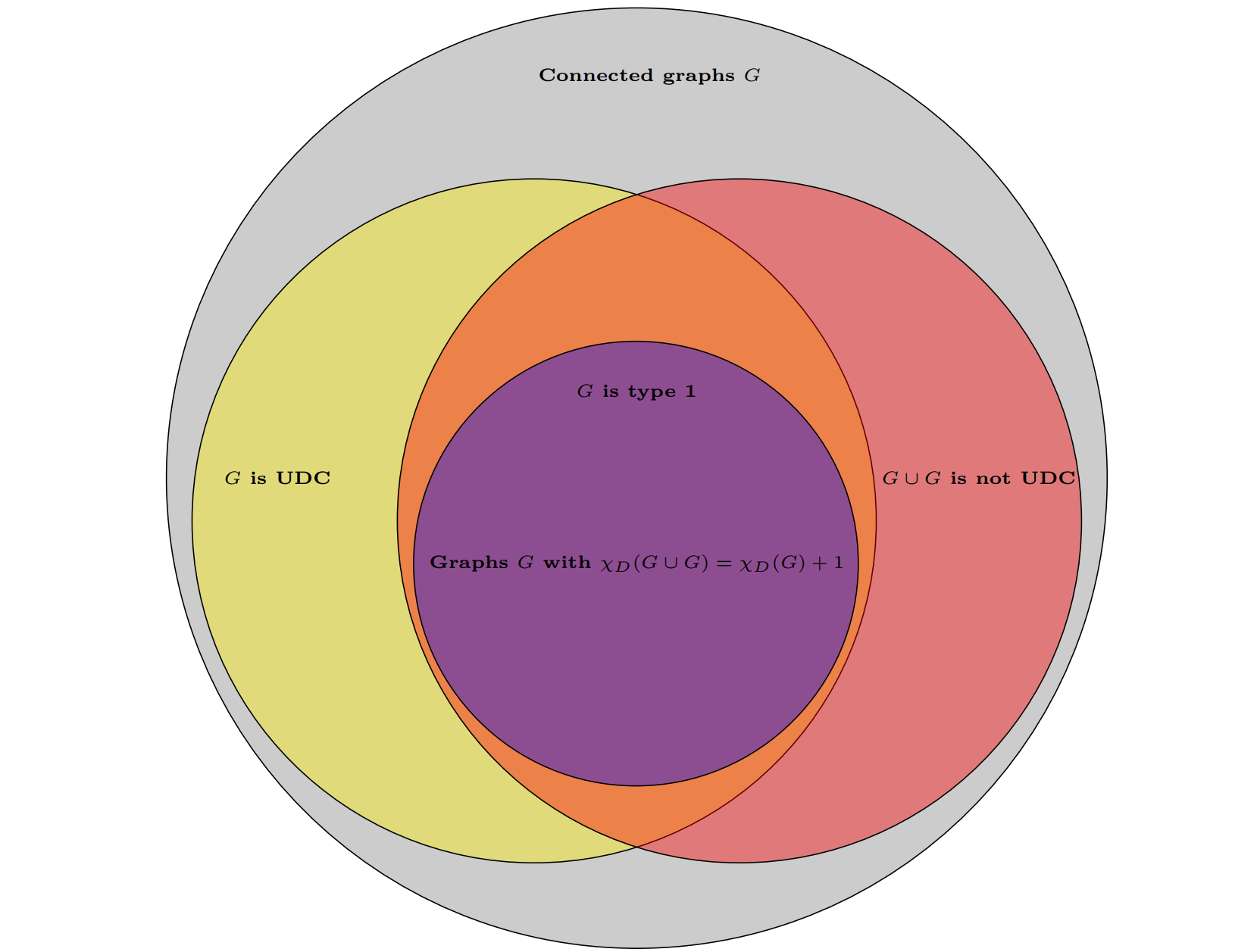}
\end{center}
\end{figure} 
In this paper, we show that the study and classification of uniquely distinguishing colorable (UDC for short) graphs help to calculate the distinguishing chromatic number of disconnected graphs.  
In Section 2, we show that  the concepts of unique colorability and unique distinguishing colorability in a graph are different. However, these two concepts are not unrelated. We have determined the unique colorability of the distinguishing coloring for some known graphs.
In Section 3, we define two types of uniquely distinguishing colorable graphs, named type 1 and type 2, and prove that 
if $ G $ is a non-empty uniquely distinguishing colorable graph, then $ G $ is connected or isomorphic to the union of two isomorphic graphs of type 2.
In Section 4, we study the properties and classification of uniquely distinguishing colorable bipartite graphs and specifically trees. 
Moreover,  we show that any uniquely distinguishing $n$-colorable tree with $ n \geq 3$ is a star graph. 
In Section 5, we prove that if 
$G$ 
is a connected graph, then 
$\chi_D(G\cup G)=\chi_D(G)+1$ 
if and only if
$G$ 
is uniquely distinguishing colorable of type 1. 
In this section, we investigate the family of connected graphs $G$ that $\chi_{D}(G\cup G) = \chi_{D}(G)+1 = 3$.
Moreover, we characterize all graphs 
$G$ of order $n$ with property that $\chi_{D}(G\cup G) = \chi_{D}(G) = k$, where $k=n-2, n-1, n$. 
Also, we determine all graphs 
$G$ of order $n$ with the property that $\chi_{D}(G\cup G) = \chi_{D}(G)+1 = \ell$, where $\ell=n-1, n, n+1$. 
The results obtained for connected graphs are given above in the form of a Venn diagram. 

%-----------------------------------------------------------------------
\section{General results}
First, we provide  two examples which show  that the concepts of unique colorability  and  unique distinguishing colorability in a graph are different. To this end, consider the graph $ G $
in Figure 1. It  is a connected bipartite graph, and so it is uniquely $ 2 $-colorable. However, it has two distinguishing colorings as follows:
\begin{align*}
  [\chi_{D}(G)]_{1} &= \{ \{a, d \}, \{b, e\}, \{c\} \},  \\
[ \chi_{D}(G)]_{2}&= \{ \{a, d \}, \{c, e\}, \{b\} \}.
\end{align*}
Hence it is not uniquely distinguishing colorable. 
\begin{figure}[h!]
\centering
\begin{tikzpicture}[scale=.8, thick, vertex/.style={scale=.5, ball color=black, circle, text=white, minimum size=.2mm},
Ledge/.style={to path={
.. controls +(45:2) and +(135:2) .. (\tikztotarget) \tikztonodes}}
]
\node[label={[label distance=.1cm]180:$a$‌}] (a1) at (180:1) [vertex] {};‌
\node[label={[label distance=.1cm]90:$b$‌}] (a2) at (90:1) [vertex] {};‌
\node[label={[label distance=.1cm]90:$d$‌}] (a3) at (0:1) [vertex] {};‌
\node[label={[label distance=.1cm]-90:$c$}] (a4) at (-90:1) [vertex] {};‌
\node[label={[label distance=.1cm]90:‌$e$}] (a5) at (0:2.4) [vertex] {};‌
\draw(a1) to(a2) to (a3) to (a4) to (a1) ;
\draw(a3) to (a5);
\end{tikzpicture}
\begin{center}
%Figure 1.
\caption{\ A uniquely colorable graph $G$ that is not uniquely distinguishing colorable.}
\end{center}
\end{figure}
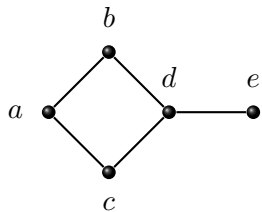

Next, consider the graph $H$ in Figure 2. Clearly
\begin{align*}
 \left[  \chi(H) \right]_{1} &= \{ \{a_{1}, a_{3} \}, \{a_{2}, a_{5}\}, \{a_{4}, a_{6}\} \}, \ \T{and} \\
 \left[  \chi(H) \right]_{2} &=  \{ \{a_{1}, a_{6} \}, \{a_{3}, a_{4}\}, \{a_{2}, a_{5}\} \} = \left[  \chi_{D}(H) \right]. 
\end{align*}
It is easy to see that $ H $ has the unique distinguishing color class $ \left[  \chi_{D}(H) \right] $. Hence,  
$ H $
is uniquely distinguishing colorable but it is not uniquely colorable.
\begin{center}
\begin{figure}[h!]
\centering
\begin{tikzpicture}[scale=.8, thick, vertex/.style={scale=.5, ball color=black, circle, text=white, minimum size=.2mm},
Ledge/.style={to path={
.. controls +(45:2) and +(135:2) .. (\tikztotarget) \tikztonodes}}
]
\node[label={[label distance=.1cm]-90:$a_{4}$‌}] (a1) at (0,0) [vertex] {};‌
\node[label={[label distance=.1cm]-90:$a_{5}$‌}] (a2) at (2,0) [vertex] {};‌
\node[label={[label distance=.1cm]-90:$a_{6}$‌}] (a3) at (4,0) [vertex] {};‌
\node[label={[label distance=.1cm]90:$a_{1}$‌}] (a4) at (0,2) [vertex] {};‌
\node[label={[label distance=.1cm]90:$a_{2}$‌}] (a5) at (2,2) [vertex] {};‌
\node[label={[label distance=.1cm]90:$a_{3}$‌}] (a6) at (4,2) [vertex] {};‌
\draw(a4) to(a1) to (a2) to (a3) ;
\draw(a4) to (a5) to (a6);
\draw(a5) to (a1);
\draw(a5) to (a3);
\draw(a6) to (a3);
\end{tikzpicture}
\begin{center}
%Figure 2.
\caption{\ A uniquely distinguishing colorable graph $H$ that is not uniquely colorable.}
\end{center}
\end{figure}
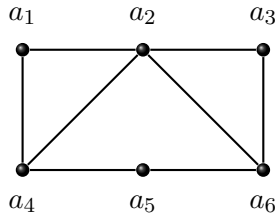
\end{center}

\begin{remark}\label{rem1}
\begin{itemize}
\item[(1)]  For positive integers $n$ and $m$, the graphs $ K_{n} $,
 $ \overline{K_{n}}$, 
 $ K_{n,m} $ and $ P_{2n} $ are uniquely distinguishing colorable. But $ C_{n} $ (for  $ n \geq 5 $) and  $ P_{2n+1} $ (for  $ n \geq 2 $) are not uniquely distinguishing colorable.
\item[(2)]  Let  $ G $ be a uniquely colorable graph  and  $ D(G) = 1$. Then $ G $ is uniquely distinguishing colorable and 
 $ \chi_{D}(G) = \chi (G) $.
 \item[(3)] Let  $ G $ be a connected graph and  $ \chi_{D}(G) = 2$. Then $ G $ is a uniquely distinguishing colorable.
\end{itemize}
\end{remark}
 
 \begin{proposition}
 Let  $G$ be a uniquely distinguishing $n$-colorable graph and 
 $n \neq \vert V(G) \vert$. Then  $G$ is uniquely distinguishing colorable with $\chi_{D}(G) = n$.
 \end{proposition}
 \begin{proof}
 Clearly $  \chi_{D}(G)  \leq n < \vert V(G) \vert $. Assume that 
 $  \chi_{D}(G) < n $  and seek a contradiction. Suppose that 
 $[\chi_{D}(G)] = \{V_{1}, V_{2}, \ldots , V_{\chi_{D}(G)}\} $ and 
choose the set 
 $ \{x_{1}, x_{2}, \ldots ,x_{\chi_{D}(G)}\}$
of vertices of $G$ such that $ x_{i} \in V_{i}$, for each $i$, $1\leq i\leq X_{D}(G)$. 
Clearly there are at least two subsets of $V(G)$ of cardinality
$ n - \chi_{D}(G) $
from the remaining vertices.
Now, each time we can color one of these subsets with 
$ \chi_{D}(G)+1, \chi_{D}(G)+2, \ldots , n  $
and  the remaining vertices with their colors given in 
$ \left[  \chi_{D}(G) \right]$. 
So we obtain at least two different colorings with $n$ colors of $G$, which is a contradiction. 
Therefore $  \chi_{D}(G) =n$.
\end{proof} 

\section{Disconnected graphs} 
In 1968, Cartwright and Harary \cite{Cartwright} showed that the study of unique colorability is limited to connected graphs. In fact, they proved that every uniquely colorable graph is connected. In this section, we show that there exists some family of disconnected uniquely distinguishing colorable graphs, and unlike proper coloring, the study of unique distinguishing colorability of disconnected graphs is challenging.

Let  $ G $ be a uniquely distinguishing colorable graph with the unique coloring
$ \left[  \chi_{D}(G) \right] = \{V_{1}, V_{2}, \ldots ,V_{n}\} $. Clearly,  in a  distinguishing $n$-coloring of  $G$, one can assign $n$ colors to $V(G)$ with $n!$ different ways. We call each of these assigning ways a labeling of the distinguishing color classes. 
For the sake of convenience, in the following definition, we introduce two types of uniquely distinguishing colorable graphs.
\begin{definition}
Let  $ G $ be a connected uniquely distinguishing colorable graph. 
We say that $G$ has type 1 if for any labeling of the distinguishing color classes of components of
$G \cup G$ 
with 
$\chi_{D}(G)$ 
colors, 
there exists an automorphism which embeds one component into the other and preserves all the distinguishing color class labels. Otherwise, we say that  
$G$
is of type 2.
\end{definition}
For instance, the graph in Figure 3(a) is of type 2, and the graph in Figure 3(b) is of type 1.

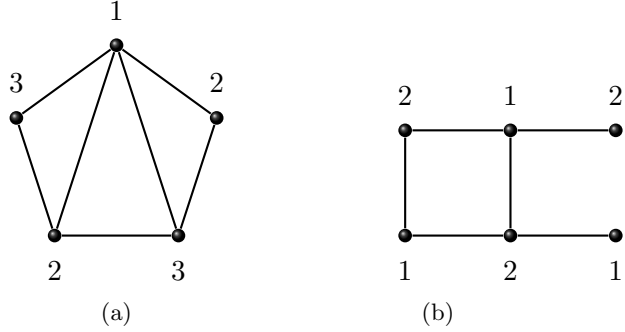
\begin{figure}[h!] \hspace{35mm}
\subfloat[]{\begin{tikzpicture}[scale=.7,thick, vertex/.style={scale=.5, ball color=black, circle, text=white, minimum size=.2mm},
Ledge/.style={to path={
.. controls +(45:2) and +(135:2) .. (\tikztotarget) \tikztonodes}}
]
\node[label={[label distance=.1cm]90:$ 2 $‌}] (a1) at (18:2) [vertex] {};‌
\node[label={[label distance=.1cm]90:$ 1 $‌}] (a2) at (90:2) [vertex] {};‌
\node[label={[label distance=.1cm]90:$ 3 $}] (a3) at (162:2) [vertex] {};‌
\node[label={[label distance=.1cm]-90:$ 2 $‌}] (a4) at (234:2) [vertex] {};‌
\node[label={[label distance=.1cm]-90:$ 3 $‌}] (a5) at (306:2) [vertex] {};‌
\draw(a1) to (a2) to (a3) to (a4) to (a5) to (a1);
\draw(a4) to (a2) to (a5);
\end{tikzpicture}}
\subfloat[]{
\hspace{18mm}
\begin{tikzpicture}[scale=.7,thick, vertex/.style={scale=.5, ball color=black, circle, text=white, minimum size=.2mm},
Ledge/.style={to path={
.. controls +(45:2) and +(135:2) .. (\tikztotarget) \tikztonodes}}
]
\node[label={[label distance=.1cm]-90:$1$‌}] (a1) at (0,0) [vertex] {};‌
\node[label={[label distance=.1cm]-90:$ 2 $‌}] (a2) at (2,0) [vertex] {};‌
\node[label={[label distance=.1cm]-90:$ 1 $‌}] (a3) at (4,0) [vertex] {};‌
\node[label={[label distance=.1cm]90:$ 2 $‌}] (a4) at (0,2) [vertex] {};‌
\node[label={[label distance=.1cm]90:$ 1 $‌}] (a5) at (2,2) [vertex] {};‌
\node[label={[label distance=.1cm]90:$ 2 $‌}] (a6) at (4,2) [vertex] {};‌
\draw(a4) to(a1) to (a2) to (a3) ;
\draw(a4) to (a5) to (a6);
\draw(a2) to (a5);

\end{tikzpicture}
}
\begin{center}
\caption{\ Uniquely distinguishing colorable graphs of of type 2 and type 1.}
\end{center}
\end{figure} 
We say that two colorings 
$\left[  \chi_{D}(G) \right]_{1}$ 
and 
$\left[  \chi_{D}(G) \right]_{2}$ 
are isomorphic, if we color each component of 
$G \cup G$ 
with one of 
$\left[  \chi_{D}(G) \right]_{1}$ 
and 
$\left[  \chi_{D}(G) \right]_{2}$, 
for any arbitrary labeling of  color classes of
$\left[  \chi_{D}(G) \right]_{1}$,  
there is a labeling of  color classes of
$\left[  \chi_{D}(G) \right]_{2}$,  
and an automorphism such that images one component of 
$G \cup G$ 
into the other and preserves all the distinguishing color class labels. 

For an integer 
$n$, 
let 
$\mathop{\cup}\limits^{n} G$
denote the disjoint union of 
$n$ 
copies of 
$G$.

\begin{lemma}\label{lem1}
Let 
$G$ 
be a connected graph. For a positive integer 
$n$, 
if 
$\mathop{\cup}\limits^{n} G$ 
is a uniquely distinguishing colorable graph, then 
$G$ is uniquely distinguishing colorable.
\end{lemma} 
\begin{proof} 
Let 
$n\geq 2$. 
If 
$G$ 
is not uniquely distinguishing colorable, then it has at least two partitions of its vertices, as 
$\left[  \chi_{D}(G) \right]_{1}$
and 
$\left[  \chi_{D}(G) \right]_{2}$, 
into distinguishing colorings. 
If 
$\left[  \chi_{D}(G) \right]_{1}$
and 
$\left[  \chi_{D}(G) \right]_{2}$ 
are not isomorphic, then we can consider two components of 
$\mathop{\cup}\limits^{n} G$ 
and color each component with one of the distinguishing colorings of $G$ each time. 
Therefore, a different partition of 
$\mathop{\cup}\limits^{n} G$  
is achieved each time, which forms a distinguishing coloring.
If 
$\left[  \chi_{D}(G) \right]_{1}$
and 
$\left[  \chi_{D}(G) \right]_{2}$ 
are isomorphic, then consider the distinguishing coloring of 
$\mathop{\cup}\limits^{n} G$. 
If both 
$\left[  \chi_{D}(G) \right]_{1}$
and 
$\left[  \chi_{D}(G) \right]_{2}$ 
are used in the distinguishing coloring of components of 
$\mathop{\cup}\limits^{n} G$, 
we do the same as above. 
Note that in this case there is a labeling of distinguishing color classes of 
$\left[  \chi_{D}(G) \right]_{1}$
and 
$\left[  \chi_{D}(G) \right]_{2}$ 
such that there is no color-preserving automorphism that image 
$\left[  \chi_{D}(G) \right]_{1}$
to 
$\left[  \chi_{D}(G) \right]_{2}$. 
Otherwise, at most one of 
$\left[  \chi_{D}(G) \right]_{1}$
and 
$\left[  \chi_{D}(G) \right]_{2}$ 
is used in the distinguishing coloring of components of
$\mathop{\cup}\limits^{n} G$. 
Assume that 
$\left[  \chi_{D}(G) \right]_{1}$
is not used in the distinguishing coloring of 
$\mathop{\cup}\limits^{n} G$. 
Consider an arbitrary labeling of distinguishing color classes of 
$\left[  \chi_{D}(G) \right]_{1}$. 
There is at most one component of 
$\mathop{\cup}\limits^{n} G$ 
such that its coloring which image with an automorphism to the 
$\left[  \chi_{D}(G) \right]_{1}$ 
with its considered labeling and preserves the labels.
Replace the coloring of this component with  
$\left[  \chi_{D}(G) \right]_{1}$ 
and its considered labeling.
Thus we have another distinguishing coloring for 
$\mathop{\cup}\limits^{n} G$.
This means that 
$\mathop{\cup}\limits^{n} G$ 
has at least two distinguishing colorings and the result follows. 
\end{proof}

\begin{lemma}\label{lem2}
Let 
$G$ 
be a connected graph.
For a positive integer 
$n$, 
if 
$\mathop{\cup}\limits^{n} G$ 
is a uniquely distinguishing colorable, then 
$\chi_{D}(\mathop{\cup}\limits^{n} G)=\chi_{D}(G)$.
\end{lemma} 
\begin{proof}
For the sake of contradiction, assume that 
$\chi_{D}(\mathop{\cup}\limits^{n} G)>\chi_{D}(G)$. 
We claim that there exist two components of 
$\mathop{\cup}\limits^{n} G$ 
with at least one different color in the distinguishing coloring of 
$\mathop{\cup}\limits^{n} G$. 
For this purpose, suppose that exactly 
$\chi_{D}(\mathop{\cup}\limits^{n} G)$ 
colors are used in each component. 
Consider the coloring of one component of 
$\mathop{\cup}\limits^{n} G$
and replace it with the unique distinguishing coloring of graph 
$G$. 
Therefore, we will have another distinguishing coloring for 
$\mathop{\cup}\limits^{n} G$. 
Now, we can select two components 
$G_1$ 
and 
$G_2$ 
of 
$\mathop{\cup}\limits^{n} G$
such that the colors used in them are not the same.  Interchanging the distinguishing colorings of 
$G_1$ 
and 
$G_2$, 
gives us a new distinguishing coloring for 
$\mathop{\cup}\limits^{n} G$, 
a contradiction.
\end{proof}
\begin{lemma}\label{lem3}
Let 
$G$ 
be a graph of type 1. Then 
$\chi_{D}(G \cup G) =\chi_{D}(G)+1$.
\end{lemma}
\begin{proof}
Since 
$G$ 
is uniquely distinguishing colorable, there is only one partition of the vertex set
$V(G)$ 
into a distinguishing coloring. Because $G$ is of type 1, it follows that for any two labelings of this partition,
there exists an automorphism of $G$ which embeds one of them into the other and preserves all the distinguishing color class labels.
Hence, we need a new color to distinguish any two components of 
$G$.
\end{proof}

\begin{lemma}\label{lem4}
Let 
$G$ 
be a connected graph. For a positive integer 
$n$, 
if 
$\mathop{\cup}\limits^{n} G$ 
is a uniquely distinguishing colorable graph, then 
$n\leq 2$.
\end{lemma} 
\begin{proof}
For the sake of contradiction, assume that 
$n\geq 3$. 
Let 
$G_1$, 
$G_2$ 
and 
$G_3$ 
be some components of 
$\mathop{\cup}\limits^{n} G$. 
Lemmas \ref{lem1}, \ref{lem2}, conclude that the colorings of 
$G_1$ 
and 
$G_2$ 
are the same with a different assigned label for some color classes. 
Let 
$V$ 
denote a color class of 
$\left[  \chi_{D}(G) \right]$ 
that is labelled by 
$1$ 
and 
$2$ 
in 
$G_1$ 
and 
$G_2$, 
respectively. Interchanging the labels of the color classes of 
$G_1$ 
and 
$G_2$, 
and keeping the coloring of 
$G_3$, 
give us another partition of vertices 
$\mathop{\cup}\limits^{n} G$ 
into a distinguishing coloring, a contradiction.
\end{proof}

\begin{theorem} \label{01}
Let $ G $ be a non-empty uniquely distinguishing colorable graph. Then $ G $ is connected or it is isomorphic to $H\cup H$, where $H$ is a connected graph of type $2$.
\end{theorem}
\begin{proof}
Let  $ G $ be a disconnected uniquely distinguishing colorable graph.  By Lemmas \ref{lem1}, \ref{lem4}, $G$ has two uniquely distinguishing colorable components, say 
$ H_{1} $ and $ H_{2}$. 
Consider a distinguishing coloring of $G$ with minimum number of colors. Then
$H_{1}$
and
$ H_{2} $
have at least one same color. 
Now suppose that $ H_{1} \ncong H_{2}$. Since $ G $ has at least one edge, we may assume that $ \chi_{D}(H_{1}) \geq 2$.
Thus $ H_{1} $ has at least two color classes  $ V_{1} $ and $ V_{2} $ in  $ \left[  \chi_{D}(H_{1}) \right]$.
Also, let $ U $ be a color class of $ H_{2}$. So we may assume that $ V_{1} $ and $ U$
have a common color. Interchange the colors 
$ V_{1} $
and 
$ V_2$,
and produce a new partition of $ V(G) $
into distinguishing color classes. This is the required contradiction. Hence we have $ H_{1} \cong H_{2}$. Now assume that $H_1$ (and so $H_2$) is of type 1. 
Lemmas \ref{lem2}, \ref{lem3} conclude that
$H_1 \cup H_2$ is not uniquely distinguishing colorable. This implies that $ G $ is a union of copies of a graph of type 2.
\end{proof}
The disconnected graph in Figure 4 is uniquely distinguishing $2$-colorable.  Clearly, each of its components  is of type 2.

\begin{figure}[h!] \hspace{22mm}
\subfloat{
\begin{tikzpicture}[scale=.7, thick, vertex/.style={scale=.5, ball color=black, circle, text=white, minimum size=.2mm},
Ledge/.style={to path={
.. controls +(45:2) and +(135:2) .. (\tikztotarget) \tikztonodes}}
]
\node[label={[label distance=.1cm]-90:$2$‌}] (a1) at (0,0) [vertex] {};‌
\node[label={[label distance=.1cm]-90:$1$‌}] (a2) at (2,0) [vertex] {};‌
\node[label={[label distance=.1cm]-90:$2$‌}] (a3) at (4,0) [vertex] {};‌
\node[label={[label distance=.1cm]90:$1$}] (a4) at (0,2) [vertex] {};‌
\node[label={[label distance=.1cm]90:‌$2$}] (a5) at (2,2) [vertex] {};‌
\node[label={[label distance=.1cm]90:‌$1$}] (a6) at (4,2) [vertex] {};‌
\node[label={[label distance=.1cm]90:‌$2$}] (a7) at (6,2) [vertex] {};‌
\draw(a4) to(a1) to (a2) to (a3) ;
\draw(a4) to (a5) to (a6)to (a7);
\draw(a2) to (a5);
\end{tikzpicture}
}
\hspace{10mm}
\subfloat{
\begin{tikzpicture}[scale=.7,thick, vertex/.style={scale=.5, ball color=black, circle, text=white, minimum size=.2mm},
Ledge/.style={to path={
.. controls +(45:2) and +(135:2) .. (\tikztotarget) \tikztonodes}}
]
\node[label={[label distance=.1cm]-90:$1$‌}] (a1) at (0,0) [vertex] {};‌
\node[label={[label distance=.1cm]-90:$2$‌}] (a2) at (2,0) [vertex] {};‌
\node[label={[label distance=.1cm]-90:$1$‌}] (a3) at (4,0) [vertex] {};‌
\node[label={[label distance=.1cm]90:$2$}] (a4) at (0,2) [vertex] {};‌
\node[label={[label distance=.1cm]90:‌$1$}] (a5) at (2,2) [vertex] {};‌
\node[label={[label distance=.1cm]90:‌$2$}] (a6) at (4,2) [vertex] {};‌
\node[label={[label distance=.1cm]90:‌$1$}] (a7) at (6,2) [vertex] {};‌
\draw(a4) to(a1) to (a2) to (a3) ;
\draw(a4) to (a5) to (a6)to (a7);
\draw(a2) to (a5);
\end{tikzpicture}
}
\begin{center}
\caption{\ A disconnected uniquely distinguishing colorable graph.}
%Figure 4.
\end{center}
\end{figure}
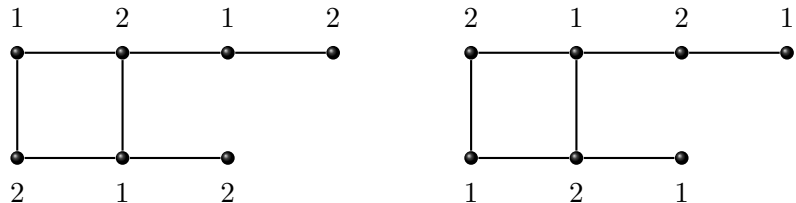
\begin{corollary} Let $ G $ be a   uniquely distinguishing colorable  graph. If $ G $ is 2-regular, then
$ G$ is isomorphic to $C_{3}$ or
$ C_{4}$.
\end{corollary}
\begin{proof}
Assume that we assign a distinguishing coloring to $G$ with minimum number of colors.
Clearly, each $2$-regular graphs is a union of some cycles. On the other hand,  it is easy to see that the uniquely distinguishing colorable cycles are 
$ C_{3} $ and $ C_{4}$. Now, let $ G $ be  a $ 2 $-regular uniquely distinguishing  colorable graph. 
We observe that $C_3$ and $C_4$ are of type 1 and type 2, respectively. If $G$ is disconnected, it
cannot be a union of some copies of $ C_{3} $,  by Theorem \ref{01}.
Suppose that  $G$ is a union of some copies of $ C_{4}$ and  that one component of $ G $ contains vertices  $ v_{1}, v_{2}, v_{3}$, $ v_{4} $  and  edges    $ v_{1}v_{2}, v_{2}v_{3}, v_{3}v_{4}$, $ v_{4}v_{1} $. Now, by replacing  the color of $ v_{1}$ with $v_{2}, v_{2}$ with $v_{3}, v_{3}$ with  $v_{4}$ and finally $ v_{4}$ with $v_{1} $,
 one can obtain  another distinguishing color class of 
 $ G $.
This implies that 
$ G $ is isomorphic to $C_{3}$ or $ C_{4} $.
\end{proof} 

The converse of Theorem \ref{01} is not correct. Actually, type 2 graphs exist in such a way that their union cannot be uniquely distinguishing colorable. 
For instance, consider an $m$-partite graph 
$G$ 
of order 
$n$ 
with
$2\leq m < n$. 
Thus, 
$G$ 
has at least a part 
$U$
with cardinality greater than one. Also,  
$G$ 
is a graph of type 2. Let
$v_1 , v_2 \in U$
be two vertices with colors 
$1$ 
and
$2$, respectively. 
Interchanging the colors 
$1$ 
and 
$2$ 
of the vertices produces a new partition of 
$V (G \cup G)$ 
into distinguishing color classes.
\begin{proposition}
Let 
$G$ 
be a connected graph.
If
$\mathop{\cup}\limits^{2} G$
is a uniquely distinguishing colorable, then for any colored graph 
$G$, 
there exists an automorphism 
$f$ 
and exactly one component 
$G'$
of 
$\mathop{\cup}\limits^{2} G$ 
such that 
$f(G)=G'$ 
and 
$f$ 
preserves the distinguishing  color class labels.
\end{proposition}
\begin{proof}
For a contradiction, assume that there exists a colored graph 
$G$ 
such that any automorphism 
$f$ 
with 
$f(G)= G'$, 
do not preserve the color class labels. By replacing the coloring of one component of 
$\mathop{\cup}\limits^{2} G$
with the distinguishing coloring of 
$G$, 
we will obtain another distinguishing coloring of 
$\mathop{\cup}\limits^{2} G$, 
a contradicting.
\end{proof}

\begin{corollary}
Let 
$G$ 
be a connected graph.
If
$G \cup G$
is a uniquely distinguishing colorable, then there exists exactly two labeling of the distinguishing color classes of 
$G$ 
that don't image together by a color-preserving automorphism.  
\end{corollary}

Let 
$G$ 
be a bipartite graph.
If 
$G$ 
is an asymmetric graph (that is, the automorphism group of 
$G$ 
is trivial), then 
$G$ 
is of type 2, 
$G \cup G$ 
is a uniquely distinguishing colorable graph and so 
$\chi_{D}(G \cup G)= \chi_{D}(G)=2$. 
This means that there are an infinite number of disconnected uniquely distinguishing colorable graphs.

\section{Bipartite graphs}
In this section, we study the uniquely distinguishing colorable bipartite graphs and specially uniquely distinguishing colorable trees. 
As main result of this section, we show that the only uniquely distinguishing colorable tree is star graph, when distinguishing chramatic number is more than 
$2$. 

\begin{lemma}\label{1}
Let $ G $ be a graph and $ x $
be a fixed  vertex of $G$ such  that $ \chi_{D}(G) - 1 >\T{deg}( x )$. Then $ G $ is not uniquely distinguishing colorable.
\end{lemma}
\begin{proof}
Let $[\chi_{D}(G)] = \{V_{1}, V_{2}, \ldots ,V_{\chi_{D}(G)} \} $
be a distinguishing coloring of $ G $ such that $ N(x) $ 
has no intersection with at least two distinguishing color classes of $G$, say with $ V_{1} $ and $ V_{2}$. Now, assume that  the distinguishing color class of $ x $ is of cardinality one and that $ x \in V_{1} $. 
By replacing  the color of $ V_{1} $ with the color of $ V_{2} $, one can obtain a distinguishing coloring of 
$ G $ with less than $ \chi_{D}(G) $ colors. This implies that  the distinguishing color class of $ x $
is of cardinality at least two. Suppose that 
$ x \in V_{1} $  and that $ \vert V_{1} \vert \geq 2$. Now, by replacing the color of
$x$ with the color of $V_{2}$, we get another distinguishing color class of 
$ G $. Hence $ G $ is not uniquely distinguishing colorable.
\end{proof}
\begin{proposition}
Let $ G $ be a connected uniquely distinguishing colorable graph with a fixed pendant vertex. Then 
$ D(G) = 1 $ and $ \chi_{D}(G) = 2 $.
\end{proposition}
\begin{proof}
By Lemma \ref{1},  
$ \chi_{D}(G) = 2$. Hence $ G $ is bipartite. Assume that $ D(G) \neq 1$ and seek a contradiction.  
There is a non-trivial automorphism 
$f$ 
of 
$G$ 
such that 
$f(v)=u$, 
for two vertices 
$v$ 
and 
$u$. 
Vertices 
$v$ 
and 
$u$ 
should be of the same distance of the pendant vertex. This implies that the distance between 
$v$ 
and 
$u$ 
must be even and so 
they will be in the same part in the bipartite graph
$G$. 
Therefore, 
$f$ 
preserve the unique distinguishing coloring and which is a contradiction. 
\end{proof}
\begin{lemma} \label{aa0}
Let  $ G $ be a uniquely distinguishing colorable graph with at least two sibling pendant vertices. Then $ G $ is a star graph.
\end{lemma}
\begin{proof}
Let $ a $ and $ b $ be sibling pendant vertices of $ G$. Assume that $ G$ is not a star graph.  Consider a distinguishing coloring of $ G$. If at least one of $ a $ and $ b $ belongs to a non-singleton color class, by interchanging the colors of $ a $ and $ b $  we get another distinguishing color class of $ G$, a contradiction. Suppose $ a $ and $ b $ belong to a singleton color class. Since $ G$ is not a star graph, there exists a vertex $c$ such that $ a \nsim c \nsim b$, and $ c $ is neither a sibling of $a$ nor $b$. Now, by replacing the color of $c$ with the color of $a$, we get another distinguishing color class of $G$, so we reach a contradiction. 
\end{proof}

Let us recall the following interesting result about trees, which we shall use in the rest of the paper.
\begin{theorem} \label{2}
\cite[Theorem 3.1]{Collins} Let $ T $ be a tree with at least $ 2 $ vertices. Then $ \chi_{D}(T) = 2 $ if and only if  $ T $ has no non-trivial automorphism, or the center of $ T $ is an edge $ xy $ and $ T $ has precisely one non-trivial automorphism, and this automorphism interchanges $ x $ and $ y$. 
\end{theorem}
\begin{theorem}
Let $ T $ be a uniquely distinguishing colorable tree with $ \chi_{D}(T) \geq 3 $. Then $ T $ is a star graph.
\end{theorem}
\begin{proof}
Assume that  $ T $ is not a star, and seek a contrdiction.  Consider a distinguishing coloring
of $T$ with minimum number of colors. By Lemma \ref{aa0}, $T$  cannot have sibling pendant vertices.

Now, we claim that there is at least one pendant vertex fixed by any automorphism of $T$ which is the required contradiction. To achieve this,  suppose that any  given  pendant vertex of $T$ is moved by  an automorphism of $T$ and  that $ v $ and $ u $ are  two pendant vertices of $T$ imaged to each other by some automorphism of $T$. Consider the unique path from  $ v $  to $ u$. 
If the length  of this path is even, then there exists a vertex, say $ a$, in the unique path from  $ v $ to $ u$ with the same distance from of $ v $ and $ u$. Among the set of automorphisms of $T$ which  image $ v $ to $ u$, consider an  automorphism $ f $ with maximal set of  fixed vertices of $T$.  
Let $A$ denote the set of fixed vertices of $T$ by $f$.
Clearly $A\neq \emptyset$, because $ a\in A $. Assume that $ A $ and $ V(T)\setminus A $ have vertices with the same color 
in the distinguishing coloring of $ T $. Now, consider the induced  subgraph $ T[V(T)\setminus A]$ containing two isomorphic components. By interchanging the colors of the corresponding vertices in the two components,  we get another distinguishing color class of $ T$, a contradiction. This implies that  the vertices in  $ A $ and the vertices in $ V(T)\setminus A$ have different colors. Since $ A $ is non-empty, there is at least one color of vertices of $A$ that is not assigned to any vertices of $ V(T)\setminus A$.  Now, we claim that, by replacing the color $ v$ (or $ u $) with this color,  one can obtain  another distinguishing color class of $ T$. To achieve this, it suffices to show that  
there is no automorphism of $T$ that moves $ v$ (or $ u $) which  preserves the vertex colors. 
Assume that, by some automorphism of $T$ which preserves the vertex colors, $ v$ (or $ u $) is imaged to some vertex of $ T$, say to $ b$,  and seek a contradiction.  Clearly $ b\in A$. Now, since $T$ has no sibling pendant vertices, $ a \nsim b $. Therefore,  $ b $ and $ v$ (or $ u $) have no common neighbors. Hence the neighbors of $ b $ (resp.  $ v$ (or $ u $)) that belongs to $ A $ (resp. $ V(T)\setminus A$) have the same color,  which is a contradiction.

Therefore the length  of the unique path from $v$ to $u$ is odd. The general form of $ T $ is depicted in Figure 8.
\begin{figure}[h!] \hspace{30mm}
\begin{tikzpicture}[thick, vertex/.style={scale=.5, ball color=black, circle, text=white, minimum size=.2mm},
Ledge/.style={to path={
.. controls +(45:2) and +(135:2) .. (\tikztotarget) \tikztonodes}}
]
\node[label={[label distance=.1cm]-90:$u‌$}] (u) at (0,0) [vertex] {};‌
\foreach  \x/\xtext in {1,2,3,4/n}
{
\node[label={[label distance=.1cm]-90:$a_\xtext‌$}] (\x) at (\x,0) [vertex] {};‌
\draw[Ledge] (\x,0) to(\x,0) node[label={[label distance=.2cm]90:‌$G_\xtext$}] {};;
}
\foreach  \x/\xtext in {5/n,6/3,7/2,8/1}
{
\node[label={[label distance=.1cm]-90:$b_\xtext‌$}] (\x) at (\x,0) [vertex] {};‌
\draw[Ledge] (\x,0) to(\x,0) node[label={[label distance=.2cm]90:‌$H_\xtext$}] {};;
}
\node[label={[label distance=.1cm]-90:$v‌$}] (v) at (9,0) [vertex] {};‌
\draw (u) to (1) to (2) to (3);
\draw[dotted] (3) to (4);
\draw (4) to (5);
\draw[dotted] (5) to (6);
\draw (6) to (7) to (8) to (v);
\end{tikzpicture}
\begin{center}
\caption{\ Graph $T$.}
\end{center}
\end{figure}
In Figure 8, $ G_{i} \cong H_{i} $ for all $i=1, \ldots , n$. Thus edge $ a_{n}b_{n} $ is the center of $ T$. Furthermore, $ a_{n}$ and $b_{n} $ are imaged to each other by an automorphism of $ T$, and $ T $ has only this non-trivial automorphism, otherwise, we obtain two pendant vertices with the same distance from $a_i$  (and $b_i$) for $i=1, \ldots , n$,  and so the distance between these two vertices is even, which
is a contradiction.  Thus  Theorem \ref{2}  implies that 
$\chi_{D}(T) = 2$, a contradiction. 
So there is at least one pendant vertex fixed by any automorphism of $ T$. Therefore, by Lemma \ref{1}, 
$ T $ is not a uniquely distinguishing colorable tree. This is a required contradiction. 
\end{proof} 

\section{Graphs with $\chi_{D}(G \cup G)= \chi_{D}(G)$}
If 
$G$ 
is a connected graph, then 
$\chi_{D}(G \cup G)= \chi_{D}(G)$ 
or 
$\chi_{D}(G \cup G)= \chi_{D}(G)+1$. 
Let 
$A, B \subseteq V(G)$.
We say that 
$A$ 
and 
$B$ 
are equal up to automorphism, if there exists an automorphism 
$f$ 
of 
$G$ 
such that 
$f(A) = B$ 
and 
$f(B) = A$.
\begin{theorem}\label{51}
Let 
$G$ 
be a connected graph. If $G$ is not uniquely distinguishing colorable, then 
$\chi_{D}(G \cup G)= \chi_{D}(G)$.
\end{theorem}
\begin{proof} 
By the way of contradiction, suppose that 
$\chi_{D}(G \cup G)= \chi_{D}(G)+1$. 
There exist two coloring 
$\left[  \chi_{D}(G) \right]_1 = \{V_{1}, V_{2}, \ldots ,V_{\chi_{D}(G)}\}$ 
and 
$\left[  \chi_{D}(G) \right]_2 = \{U_{1}, U_{2}, \ldots ,U_{\chi_{D}(G)}\}$ 
for 
$G$. 
If both components of 
$G \cup G$ 
are colored with one of 
$\left[  \chi_{D}(G) \right]_1$ 
and
$\left[  \chi_{D}(G) \right]_2$, 
or each component is colored with one of  
$\left[  \chi_{D}(G) \right]_1$ 
and
$\left[  \chi_{D}(G) \right]_2$, 
such that labels from 
$1$ 
to 
$\chi_{D}(G)$ 
are used in each component,
there will be an automorphism which embeds one component into the other and preserves all the distinguishing color class labels. 
This implies that all color classes of 
$\left[  \chi_{D}(G) \right]_1$ 
and
$\left[  \chi_{D}(G) \right]_2$, 
are equal up to automorphism. 
Since 
$\left[  \chi_{D}(G) \right]_1$ 
and
$\left[  \chi_{D}(G) \right]_2$ 
give us different partitions of 
$V(G)$, 
we can consider 
$A:=V_1 \cap U_1 \neq \emptyset$  
and 
$B:= V_1 - A = U_1 - A \neq \emptyset$. 
In each color class, there is a copy of 
$A$ 
and 
$B$. 
For any 
$i$, 
$1\leq i \leq \chi_{D}(G)$,
let
$V_i = A_{V_i } \cup B_{V_i }$ 
and 
$U_i = A_{U_i } \cup B_{U_i }$, 
where 
$A_{V_i }$ ($A_{U_i }$) 
and 
$B_{V_i }$ ($B_{U_i }$) 
are equal up to automorphism to
$A$ 
and 
$B$, 
respectively. Clearly, for any 
$i$, 
$1\leq i \leq \chi_{D}(G)$, 
$A_{V_i } \cup B_{V_i }$ 
and 
$A_{U_i } \cup B_{U_i }$
are independent sets. Also  
$A_{V_1}=A_{U_1}$,  
and 
$A_{U_1} \cup B_{V_1}$
is an independent set. 

We claim that there is no edge between 
$A_{V_i }$ 
and 
$B_{V_j }$, 
$1\leq i, j \leq \chi_{D}(G)$. 
Assume, for the sake of contradiction, that there is an edge between a vertex of 
$A_{V_j}$ 
and a vertex of
$B_{V_k}$, 
for some 
$j$ 
and 
$k$, 
$1\leq j<k \leq \chi_{D}(G)$. 
Consider a coloring of components of
$G \cup G$ 
with 
$\left[  \chi_{D}(G) \right]_1$
such a way that in the first component assign label
$i$ 
to 
$V_i$, $1\leq i \leq \chi_{D}(G)$, 
and in the second component assign label 
$k$ 
to 
$V_1$, 
label 
$1$ 
to
$V_k$ 
and labeling other color classes 
$V_i$ 
with 
$i$, $2\leq i(\neq k)\leq \chi_{D}(G)$. 
There exists an automorphism that images these two components to each other and preserves all the distinguishing color class labels. 
Thus, there is an edge between a vertex of 
$A_{V_j}$ 
and a vertex of
$B_{V_1}$.

Again, consider a coloring of the components of
$G \cup G$ 
such that each component is colored with one of the 
$\left[  \chi_{D}(G) \right]_1$ 
and
$\left[  \chi_{D}(G) \right]_2$ 
such a way that the first component is colored with 
$\left[  \chi_{D}(G) \right]_1$ 
and assign label 
$i$ 
to 
$V_i$, $1\leq i \leq \chi_{D}(G)$, 
and the second component is colored with 
$\left[  \chi_{D}(G) \right]_2$, assign label
$j$ 
to 
$U_1$, 
label
$1$ 
to
$U_j$  
and labeling other color classes 
$U_i$ 
with 
$i$, $2\leq i(\neq j) \leq \chi_{D}(G)$. 
The automorphism imaging these two components to each other and preserves all the distinguishing color class labels conclude that 
there is an edge between a vertex of 
$A_{U_1}$ 
and a vertex of
$B_{V_1}$, 
a contradiction. 

Therefore, there is no edge between 
$A_{V_i }$ 
and 
$B_{V_j }$, 
$1\leq i, j \leq \chi_{D}(G)$ 
and so 
$G$ 
is a disconnected graph, which is the required contradiction.
\end{proof} 
The following corollary is immediate by Theorem \ref{51}, Lemma \ref{lem2}, Lemma \ref{lem3} and Theorem \ref{01}.
\begin{corollary}
Let 
$G$ 
be a connected graph. Then 
$\chi_D(G\cup G)=\chi_D(G)+1$ 
if and only if
$G$ 
is uniquely distinguishing colorable of type 1.
\end{corollary}
%\begin{remark}
%Let 
%$G$ 
%be a connected graph. 

%$\chi_D(G\cup G)=\begin{cases}
%\chi_D(G), & \text{if}\; G is not uniqely distinguishing colorable\\
%\begin{cases}
%\chi_D(G), & \text{if}\; G is not uniqely distinguishing colorable\\
%3, & \text{if}\; n\geq 4.
%\end{cases}
%\end{cases}$ 
%\end{remark}

In what follows, we characterise all connected graphs 
$G$ 
with
$\chi_D(G\cup G)=\chi_D(G)+1$ 
when
$\chi_{D}(G)=2$.

\begin{lemma}\label{lem5}
Let 
$G$ 
be a graph. For two vertices 
$v_1$ 
and 
$v_2$ 
of 
$G$, 
if there exist an automorphism 
$f$ 
of 
$G$ 
with 
$f(v_1)=v_2$, 
then there exists an automorphism 
$g$ 
of 
$G$ 
with
$g(v_1)= v_2$ 
and 
$g(v_2)= v_1$. 
\end{lemma} 
\begin{proof}
Assume that the order of
$G$ 
is
$n$, 
$f^{2}(v_1)=f(f(v_1))=f(v_2)=v_3$ 
and
$f^{i}(v_1)=v_{i+1}$. 
We claim that there exists an integer
$k \leq n$ 
such that 
$f^{k}(v_1)=v_1$. 
Suppose that for any 
$k$, 
$f^{k}(v_1)\neq v_1$. 
So, let
$\ell$
be the smallest integer for which
$f^{\ell}(v_1)=v_i$ 
and 
$ i\leq \ell$.  
Thus, 
$f^{i-1}(v_1)=v_i=f^{\ell}(v_1)=f^{i-1}(f^{\ell-i+1}(v_1))$. 
This implies that 
$f^{\ell-i+1}(v_1)= v_1$, 
a contradiction. Hence, 
there are vertices 
$v_1, v_2, \ldots , v_k$ 
such that 
$f(v_i)=v_{i+1}$ 
and 
$f(v_k)=v_1$, 
for 
$1\leq i \leq k-1$.
The automorphism 
$f$ 
appears a symmetry in the graph so that we can consider an automorphism 
$g$ 
with  
$g(v_1)=v_2$, 
$g(v_2)=v_1$, 
$g(v_3)=v_k$, 
$g(v_k)=v_3$, 
and
$g(v_{k-i})=v_{3+i}$. 
\end{proof}

For a graph 
$G$ 
and 
$f \in {\rm Aut}(G)$, 
let
${\rm Fix}(f)=\{v \in V(G) \; | \; f(v)=v\}$.

\begin{theorem}
Let 
$G$ 
be a connected graph with 
$\chi_{D}(G)=2$. 
Then 
$\chi_{D}(G \cup G)=3$ 
if and only if 
$G$ 
is a bipartite graph with 
${\rm Aut}(G) \cong \mathbb{Z}_2 = \{id, f\}$, 
${\rm Fix}(f)= \emptyset$ 
and 
$f$ 
images the two parts 
of 
$G$
to each other.
\end{theorem}
\begin{proof}
$(\Rightarrow)$
By Remark \ref{rem1}(3), 
$G$ 
is uniquely distinguishing colorable. Let 
$V$ 
and 
$U$ 
denote the parts of 
$G$. 
The distinguishing coloring of 
$G$ 
is 
$\{V, U\}$. 
The automorphism group of 
$G$ 
is not trivial (since otherwise 
$\chi_{D}(G \cup G)= \chi_{D}(G)$).
Assume first that there exists an automorphism 
$f$ 
of 
$G$ 
with  
$f(v_1)= v_2$, 
for two vertices 
$v_1$ 
and 
$v_2$ 
of 
$V$. 
By Lemma \ref{lem5}, without loss of generality, we may assume that 
$f(v_2)= v_1$.
Since 
$f$ 
does not preserve the coloring, there are vertices 
$v_3 \in V$ 
and 
$u_1 \in U$ 
such that 
$f(v_3)= u_1$. 
Let 
$P$ 
be a path between 
$v_1$ 
and 
$v_2$. 
Observation 16 \cite{erwin} implies that one of vertices 
$v_3$ 
and 
$u_1$ 
cannot be in 
$P$ 
and the other be out of the path. We now divide the proof into two cases. 
\begin{itemize}
\item[{\bf Case 1.}] $v_3, u_1 \in V(P)$. 
The distance between two vertices in the same parts and different parts is even and odd, respectively. Without loss of generality, we may assume that 
${\rm d}(v_3, v_2)<{\rm d}(v_3, v_1)$ 
in 
$P$. 
Since 
$f$
flips path 
$P$, 
the 
$v_1-u_1$ path 
and 
$v_2-v_3$ path 
must be of the same size in 
$P$. 
But the size of 
$v_1-u_1$ path is odd
and 
$v_2-v_3$ path is even.

\item[{\bf Case 2.}] $v_3, u_1 \notin V(P)$. 
In this case, if there exists a path including 
$v_1$ 
and 
$v_2$ 
that 
$f$
flips this path and moves at least one of 
$v_3$ 
and 
$u_1$, 
similarity to case 1, it is impossible. Let $P'$ be a path including 
$v_1$ 
and 
$v_2$.  
Assume that when 
$P'$
flips, does not move 
$v_3$ 
and 
$u_1$. 
Thus 
$v_3$ 
and 
$u_1$ 
have the same distance from 
$v_1$ 
and 
$v_2$, 
a contradiction.
\end{itemize}
So, any automorphism image the two parts 
of 
$G$
together. Let 
$f$ 
be such an automorphism. Suppose that 
$w \in {\rm Fix}(f)$  
and 
$f(v)=u$, 
for vertices 
$v\in V$
and 
$u \in U$.
Thus, 
${\rm d}(v, w)={\rm d}(u, w)$. 
But, if
$w \in V$  ($w \in U$) 
the distance between 
$w$ 
and 
$v$
is even (odd) and the distance between 
$w$ 
and 
$u$
is odd (even). This implies that 
${\rm Fix}(f)= \emptyset$ 
and there is only an non-trivial automorphism of 
$G$. 
(Note that if there are two automorphisms 
$f$ 
and 
$g$ 
that images the two parts of 
$G$
together, then there are vertices 
$v_1 \in V$ 
and 
$u_1, u_2 \in U$ 
that 
$f(v_1)=u_1$ 
and 
$g(v_1)=u_2$. 
Therefore, 
$g(f^{-1}(u_1))=u_2$ 
and it is impossible.)

$(\Leftarrow)$ Since 
$G$ 
is a connected bipartite graph and there is an automorphism that images the two parts 
of 
$G$
together, 
$G$ 
is of type 1 and so $\chi_{D}(G \cup G)= \chi_{D}(G)+1$. 
\end{proof}

In following, we investigate all graphs $G$ 
such that 
$\chi_D(G\cup G)=\chi_D(G)$, 
when
$\chi_{D}(G)=k$, 
for
$k\in \{|V(G)|-2, |V(G)|-1, |V(G)|\}$.
To do this, we begin with Theorem \ref{e}, Theorem \ref{e0} and Theorem \ref{e1}, in which all graphs 
of order $n$ with distinguishing chromatic number $n$, $n-1$ and $n-2$ were charaterized \cite{Collins, dis 1 2}. 
We first state some necessary preliminaries.

For any graph  $G$ with vertices  $(v_1, \ldots, v_n)$ and for any collection of vertex-disjoint graphs 
$H_1, \ldots, H_n$, let $G(H_1, \ldots, H_n)$ denote the graph obtained from $G$ by replacing each
$v_i$ with a copy of $H_i$ and replacing each edge $v_i v_j$ by $H_i \vee H_j$. 
If an $H_i$ is vacuous, i.e., $H_i = \emptyset$, then replacing $v_i$ by $\emptyset$ refers to deleting 
$v_i$ and all edges incident to it.  
Note that  the substituted $H$ can be an independent set; i.e. both the empty set and independent sets are viewed as ``complete multipartite" graphs.
We present three defined graphs 
$\hat{G}_5$, $\hat{G}_6$ and $\hat{G}_7$ 
along with a class of labelled graphs {\rsfs G}$_3$ consisting of two non-isomorphic graphs. 
The labelled graphs 
$ \hat{G}_{5} $, $ \hat{G}_{6} $
and 
$ \hat{G}_{7} $ 
have vertices  
$ (v_1, v_2, v_3, v_4) $, 
$ (v_1, v_2, v_3, v_4, v_5) $ 
and 
$ (v_1, v_2, v_3, v_4, v_5, v_6) $
respectively, while a labelled graph  
$ \hat{G}_{3} $ 
belonging to the class 
{\rsfs G}$_3$  
has vertices $ (v_1, v_2, v_3, v_4, v_5)$, 
see Figure 6 and Figure 7. 
Furthermore, define 
$\hat{K}_2$ 
and 
$\hat{K}_3$ 
to be the labelled complete graphs of orders two and three respectively, where 
$\hat{K}_{2}(v_1, v_2)$ 
has vertices 
$(v_1, v_2)$ 
and 
$\hat{K}_{3}(v_1, v_2, v_3)$ 
has vertices
$(v_1, v_2, v_3)$. 
In particular, if 
$H_1$ 
and 
$H_2$ 
are nonvacuous complete multipartite graphs, then
$\hat{K}_{2}(H_1, H_2)$
represents a complete multipartite graph with at least two parts. 

\begin{theorem}\label{e}\cite[Theorem 2.3]{Collins}
Let $G$ be a graph. Then $\chi_{D}(G)=|V(G)|$ if and only of $G$ is a complete multipartite graph.
\end{theorem}

\begin{theorem}\label{e0}\cite[Theorem 3.2]{dis 1 2}
Let 
$G$ 
be a graph of order 
$n > 3$. 
Then 
$\chi_D(G) = n-1$ 
if and only if 
$G$ 
is the join of a complete multipartite graph (possibly vacuous) with one of the following: 
\begin{itemize}
\item[(1)]
$2K_2$, or  
\item[(2)] 
$H\cup K_1$, 
where 
$H$ 
is a complete multipartite graph with at least two parts.
\end{itemize}
\end{theorem}

\begin{figure}[h!] \hspace{25mm}
\subfloat{\begin{tikzpicture}[scale=.4, thick, vertex/.style={scale=.5, ball color=black, circle, text=white, minimum size=.2mm},
Ledge/.style={to path={
.. controls +(45:2) and +(135:2) .. (\tikztotarget) \tikztonodes}}
]
\node[label={[label distance=.1cm]90:$v_1$‌}] (a1) at (-2, 4) [vertex] {};‌
\node[label={[label distance=.001cm]-90:$v_2$‌}] (a2) at (-2, 0) [vertex] {};‌
\node[label={[label distance=.001cm]-90:$v_3$}] (a3) at (2, 0) [vertex] {};‌
\node[label={[label distance=.001cm]90:$v_4$}] (a4) at (2, 4) [vertex] {};‌
\draw(a1) to (a2) ;
\draw(a2) to (a3) ;
\draw(a3) to (a4) ; 
\end{tikzpicture}}
\subfloat{
\hspace{45mm}
\begin{tikzpicture}[scale=.4 ,thick, vertex/.style={scale=.5, ball color=black, circle, text=white, minimum size=.2mm},
Ledge/.style={to path={
.. controls +(45:2) and +(135:2) .. (\tikztotarget) \tikztonodes}}
]
\node[label={[label distance=.1cm]90:$v_1$‌}] (a1) at (-2, 4) [vertex] {};‌
\node[label={[label distance=.001cm]-90:$v_2$‌}] (a2) at (-2, 0) [vertex] {};‌
\node[label={[label distance=.001cm]-90:$v_3$}] (a3) at (2, -1) [vertex] {};‌
\node[label={[label distance=.001cm]90:$v_4$}] (a4) at (2, 1) [vertex] {};‌
\node[label={[label distance=.001cm]90:$v_5$}] (a5) at (2, 4) [vertex] {};‌
\draw(a1) to (a2) ;
\draw(a2) to (a3) ;
\draw(a3) to (a4) to (a2) ; 
\end{tikzpicture}
} 
\begin{center}
\caption{\  Graphs $\hat{G}_{5}(v_1, v_2, v_3, v_4)$ and $\hat{G}_{6}(v_1, v_2, v_3, v_4, v_5)$.}
\end{center}
\end{figure}
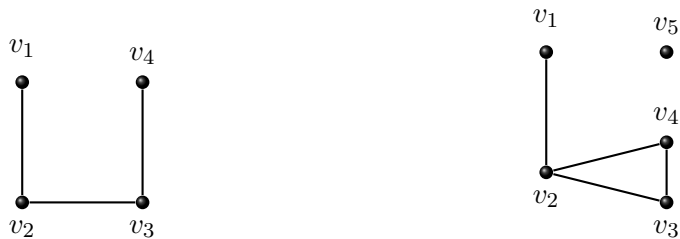 
 \begin{figure}[h!] \hspace{25mm}
\subfloat{\begin{tikzpicture}[scale=.6, thick, vertex/.style={scale=.5, ball color=black, circle, text=white, minimum size=.2mm},
Ledge/.style={to path={
.. controls +(45:2) and +(135:2) .. (\tikztotarget) \tikztonodes}}
]
\node[label={[label distance=.001cm]90:$v_1$‌}] (a1) at (-2, 1) [vertex] {};‌
\node[label={[label distance=.001cm]-90:$v_2$‌}] (a2) at (-1, 0) [vertex] {};‌
\node[label={[label distance=.001cm]-90:$v_3$}] (a3) at (1, 0) [vertex] {};‌
\node[label={[label distance=.001cm]90:$v_4$}] (a4) at (2, 1) [vertex] {};‌
\node[label={[label distance=.001cm]90:$v_5$}] (a5) at (1, 3) [vertex] {};‌
\node[label={[label distance=.001cm]90:$v_6$}] (a6) at (-1, 3) [vertex] {};
\draw(a1) to (a2) to (a3) to (a4) to (a5) to (a6) to (a1) ;
\draw(a4) to (a1) to (a3);
\draw(a4) to (a2); 
\draw(a6) to (a2); 
\draw(a5) to (a3);  
\end{tikzpicture}}
\subfloat{
\hspace{25mm}
\begin{tikzpicture}[scale=1 ,thick, vertex/.style={scale=.5, ball color=black, circle, text=white, minimum size=.2mm},
Ledge/.style={to path={
.. controls +(45:2) and +(135:2) .. (\tikztotarget) \tikztonodes}}
]
\node[label={[label distance=.1cm]90:$v_1$‌}] (a1) at (0, 0) [vertex] {};‌
\node[label={[label distance=.001cm]90:$v_2$‌}] (a2) at (1, 1) [vertex] {};‌
\node[label={[label distance=.001cm]-90:$v_3$}] (a3) at (1, -1) [vertex] {};‌
\node[label={[label distance=.001cm]90:$v_4$}] (a4) at (3, 1) [vertex] {};‌
\node[label={[label distance=.001cm]-90:$v_5$}] (a5) at (3, -1) [vertex] {};‌ ‌ 
\node[label={[label distance=.001cm]-90:$\vdots$}] (a) at (3, 0) [] {};‌
\node[label={[label distance=.001cm]-90:$\vdots$}] (a) at (3, .4) [] {};‌
\node[label={[label distance=.001cm]-90:$\vdots$}] (a) at (3, .8) [] {};‌
\node[label={[label distance=.001cm]-90:$\vdots$}] (a) at (3, 1.2) [] {};‌
\node[label={[label distance=.001cm]-90:$\vdots$}] (a) at (3, 1.6) [] {};‌
\draw(a3) to (a5) ;
\draw(a2) to (a4) ; 
\draw(a2) to (a3) ;
\draw(a3) to (a1) to (a2) ; 
\end{tikzpicture}
} 
\begin{center}
\caption{\  Graphs $\hat{G}_{7}(v_1, v_2, v_3, v_4, v_5, v_6)$  and {\rsfs G}$_3$ $(v_1, v_2, v_3, v_4, v_5)$.}\label{fig3}
\end{center}
\end{figure}
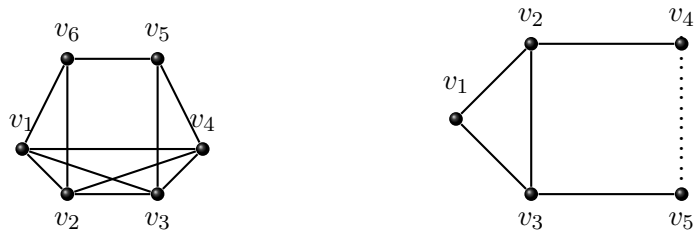 
\newpage
\begin{theorem}\label{e1}\cite[Theorem 3.5]{dis 1 2} 
Let 
$G$ 
be a graph of order 
$n > 4$. 
Then $\chi_{D}(G) = n - 2$ 
if and only if 
$G$ 
is the join of a complete multipartite graph (possibly vacuous) with one of the following: 
\begin{small}
\begin{multicols}{2}
\begin{enumerate}
\item[(a)] $ P_5 $
\item[(c)] $ C_6 $
\item[(e)] $ \overline{K_r} \cup K_2 $, for $ r \geq 2$
\item[(g)] $ \hat{K}_{3}(H_1, H_2, H_3) \cup K_2 $
\item[(i)] $ 2K_2 \vee 2K_2 $ 
\item[(k)] $ 2K_2 \cup K_1 $ 
\item[(m)]$\hat{G}_{3}(H_1, H_2, H_3, K_1, K_1)$, for $\hat{G}_{3} \in${\rsfs G}$_3$
\item[(o)] $\hat{G}_{5}(K_1, H_1, H_2, K_1)$
\item[(q)] $\hat{G}_{7}(H_1, H_2, H_3, H_4, K_1, K_1)$ 
\item[(b)] $ C_5 $ 
\item[(d)] $ 2K_3 $
\item[(f)] $\hat{K}_{2}(\overline{K_r}, H_1) \cup K_2$, for $ r\geq 2 $
\item[(h)] $\hat{K}_{2}(H_1, H_2) \cup \overline{K_2}$ 
\item[(j)] $ 2K_2 \vee (\hat{K}_{2}(H_1, H_2) \cup K_1) $
\item[(l)] $ (2K_2 \vee H_1) \cup K_1 $
\item[(n)] $\hat{G}_{5}(H_1, H_2, K_1, K_1)$
\item[(p)] $\hat{G}_{6}(K_1, H_1, H_2, H_3, K_1)$ 
\end{enumerate}
\end{multicols}
\end{small}
where each of 
$H_1, H_2, H_3, H_4$ 
is a nonvacuous complete multipartite graph. 
\end{theorem}
\begin{theorem}\label{5.8}
Let 
$G$ 
be a graph of order
$n$ 
and  
$\chi_D(G)=n$. 
Then 
\begin{center}
$\chi_D(G \cup G)=\begin{cases}
n+1, & \text{if}\; G\cong K_n\\
2n, & \text{if}\; G\cong \overline{K_n}\\
n, & \text{Otherwise.}\;
\end{cases}$
\end{center}
\end{theorem}
\begin{proof} 
By Theorem \ref{e}, 
$G$ 
is a complete multipartite graph.
The only non-trivial case is when 
$G \notin \{K_n,  \overline{K_n}\}$. 
In this case, there is at least one part, via $U$,  with cardinality more than 
$2$. 
There exist at least two subsets of size 
$|U|$ 
of 
$\{1, 2. \ldots, \chi_D(G)\}$. 
Color 
$U$ 
with 
$\{1, 2, \ldots, |U|\}$ 
in a component of 
$G \cup G$ 
and color 
$U$
in the other component with
$\{1, 2, \ldots, |U|-1, |U|+1\}$.
Also, color other vertices of 
$G \cup G$ 
arbitrary. 
One can check that this is a distinguishing $\chi_D(G)$-coloring for 
$G \cup G$.  
\end{proof}
%\newpage
\begin{lemma}\label{lem6}
Let 
$G$ 
be the join of a nonvacuous complete multipartite graph 
with one of graphs (1) and (2) in Theorem \ref{e0} or one of graphs (a) - (q) in Theorem \ref{e1}. 
Then 
$\chi_{D}(G \cup G)= \chi_{D}(G)$.
\end{lemma}
\begin{proof} 
Let 
$H$
be a nonvacuous complete multipartite graph and 
$K$
be one of graphs (1) and (2) in Theorem \ref{e0} or one of graphs (a) - (q) in Theorem \ref{e1}. 
Let 
$G_i$,
$H_i$ 
and 
$K_i$
denote the graphs isomorphic to 
$G$, 
$H$ 
and 
$K$
respectively,
for
$i=1, 2$. 
We claim that if 
$f$ 
is an automorphism with 
$f(G_1)=G_2$, 
then 
$f(H_1)=(H_2)$. 
For this, assume that all parts of
$H$ 
have cardinality more than one. On the other hand, all vertices of 
$K$ 
adjacent to all vertices of 
$H$ 
in 
$G$.
This implies that we cannot replace some vertices of 
$H$ 
with some vertices of 
$K$ 
such that the structure of 
$H$ 
is preserved. 
(Note that 
$K$ 
is not a complete multipartite graph.)
If 
$H$ 
has a part $\{x\}$, then a necessary condition for that we can replace 
$\{x\}$ 
with a vertex of 
$K$ 
such that the structure of 
$H$ 
is preserved, is that there exist a vertex 
$y \in K$
with 
${\rm deg}_{K}(x)=|V(K)|-1$. 
One can check there is no such $x$ in graphs (1) and (2) in Theorem \ref{e0} and graphs (a) - (q) in Theorem \ref{e1}.

There exist two vertices with the same color in 
$G$. 
Assume that color 
$1$ 
is assigned to two vertices in 
$G_1$ 
and color 
$2$ 
is assigned to a vertex in 
$H_1$. 
The vertices colored by $1$ are in 
$K_1$. 
Now, color the vertices of 
$G_2$ 
with the coloring of 
$G_1$ 
by interchanging colors 
$1$ 
and 
$2$. 
Therefore, any automorphism that image 
$G_1$
to 
$G_2$ 
does not preserve the colors.
\end{proof}

%The next two Theorem is immediate, by Lemma \ref{lem6}.

\begin{theorem}
Let 
$G$ 
be a graph of order 
$n > 3$ 
with
$\chi_D(G) = n-1$.
Then 
\begin{itemize}
\item[($a$)]
$\chi_D(G\cup G)=\chi_D(G)$ 
if and only if 
\begin{itemize}
\item[($a_1$)]
$G$ 
is the join of a nonvacuous complete multipartite graph with one of the following: 
\begin{itemize}
\item[($a_{11}$)]
$2K_2$, or, 
\item[($a_{12}$)]
$H\cup K_1$.
\end{itemize} 
\item[($a_2$)]
$G\cong H\cup K_1$, $H\ncong K_{n-1}$.
\end{itemize}
\item[($b$)] 
$\chi_D(G\cup G)=\chi_D(G)+1$ 
if and only if  
$G$ 
is one of the following:
\begin{itemize}
\item[($b_{1}$)]
$2K_2$ 
\item[($b_{2}$)]
$K_{n-1}\cup K_1$.
\end{itemize}
\end{itemize}
Where 
$H$ 
is a complete multipartite graph with at least two parts. 
\end{theorem}
\begin{proof}
According to Theorem \ref{e0}, assume first that 
$G$ 
is the join of a nonvacuous complete multipartite graph with one of the graphs
$2K_2$ 
or 
$H\cup K_1$. 
Lemma \ref{lem6} concludes the results in 
$(a_1)$. 
For 
$(b_1)$, 
let
$G=2K_2$. 
Hence, 
$\chi_D(G\cup G)= \chi_D(4K_2)=4=\chi_D(G)+1$ 
and the result follows.
Let 
$G=H\cup K_1$. 
Since 
$H$ 
is a complete multipartite graph with at least two parts, 
$H \ncong \overline{K_{n-1}}$. 
Now, the results in 
$(b_2)$  
and 
$(a_2)$
conclude from Theorem \ref{5.8}.
\end{proof}

%%%%%%%%%%%%%%%%%%%%%%%%%%%%%%
%\newpage
\begin{theorem}
Let 
$G$ 
be a graph of order 
$n > 4$ 
with 
$\chi_D(G) = n-2$.
Then 
\begin{itemize}
\item[($a$)]
$\chi_D(G\cup G)=\chi_D(G)$ 
if and only if 
\begin{itemize}
\item[($a_1$)]
$G$ 
is the join of a nonvacuous complete multipartite graph with one of the following: 
\begin{itemize} 
\begin{small}
\begin{multicols}{2}
\item[($a_{11}$)] $ P_5 $
\item[($a_{13}$)] $ C_6 $
\item[($a_{15}$)] $ \overline{K_r} \cup K_2 $, for $ r \geq 2$
\item[($a_{17}$)] $ \hat{K}_{3}(H_1, H_2, H_3) \cup K_2 $
\item[($a_{19}$)] $ 2K_2 \vee 2K_2 $ 
\item[($a_{111}$)] $ 2K_2 \cup K_1 $ 
\item[($a_{113}$)] $ (2K_2 \vee H_1) \cup K_1 $
\item[($a_{115}$)] $\hat{G}_{5}(K_1, H_1, H_2, K_1)$
\item[($a_{117}$)] $\hat{G}_{7}(H_1, H_2, H_3, H_4, K_1, K_1)$ 
\item[($a_{12}$)] $ C_5 $ 
\item[($a_{14}$)] $ 2K_3 $
\item[($a_{16}$)] $\hat{K}_{2}(\overline{K_r}, H_1) \cup K_2$, for $ r\geq 2 $
\item[($a_{18}$)] $\hat{K}_{2}(H_1, H_2) \cup \overline{K_2}$ 
\item[($a_{110}$)] $ 2K_2 \vee (\hat{K}_{2}(H_1, H_2) \cup K_1) $
\item[($a_{112}$)]$\hat{G}_{3}(H_1, H_2, H_3, K_1, K_1)$, for $\hat{G}_{3} \in${\rsfs G}$_3$
\item[($a_{114}$)] $\hat{G}_{5}(H_1, H_2, K_1, K_1)$
\item[($a_{116}$)] $\hat{G}_{6}(K_1, H_1, H_2, H_3, K_1)$ 
\end{multicols}
\end{small} 
\end{itemize} 
\item[($a_2$)]  
$G$ 
is one of the following: 
\begin{itemize}
\begin{small}
\begin{multicols}{2}
\item[($a_{21}$)] $ P_5 $
\item[($a_{23}$)] $ C_6 $
\item[($a_{25}$)] $ 2K_2 \vee 2K_2 $  
\item[($a_{27}$)] $ (2K_2 \vee H_1) \cup K_1 $
\item[($a_{29}$)] $\hat{G}_{7}(H_1, H_2, H_3, H_4, K_1, K_1)$ 
\item[($a_{211}$)] $\hat{G}_{6}(K_1, H_1, H_2, H_3, K_1)$
\item[($a_{22}$)] $ C_5 $ 
\item[($a_{24}$)] $ 2K_3 $
\item[($a_{26}$)] $\hat{K}_{2}(\overline{K_r}, H_1) \cup K_2$, for $ r\geq 2 $ 
\item[($a_{28}$)] $ 2K_2 \vee (\hat{K}_{2}(H_1, H_2) \cup K_1) $
\item[($a_{210}$)]$\hat{G}_{3}(H_1, H_2, H_3, K_1, K_1)$, for $\hat{G}_{3} \in${\rsfs G}$_3$ 
\end{multicols}
\item[($a_{212}$)] $ \hat{K}_{3}(H_1, H_2, H_3) \cup K_2 $, where $\hat{K}_{3}(H_1, H_2, H_3)$ is not a complete graph.
\item[($a_{213}$)] $\hat{K}_{2}(H_1, H_2) \cup \overline{K_2}$, where $\hat{K}_{2}(H_1, H_2)$ is not a complete graph and $|\hat{K}_{2}(H_1, H_2)|\geq 4$.
\item[($a_{214}$)] $\hat{G}_{5}(H_1, H_2, K_1, K_1)$, where $H_1 \neq K_1$ or $H_2 \neq K_1$. 
\item[($a_{215}$)] $\hat{G}_{5}(K_1, H_1, H_2, K_1)$, where $H_1 \neq K_1$ or $H_2 \neq K_1$. 
\end{small}
\end{itemize}
\end{itemize}
where each of 
$H_1, H_2, H_3, H_4$ 
is a nonvacuous complete multipartite graph.
\item[($b$)] 
$\chi_D(G\cup G)=\chi_D(G)+1$ 
if and only if  
$G$ 
is one of the following: 
\begin{itemize}
\begin{small}
\begin{multicols}{2}
\item[($b_{1}$)] $K_{n-2}\cup K_2$
\item[($b_{2}$)]$K \cup \overline{K_2}$, $K \in \{P_3, K_{3}\}$.
\item[($b_{3}$)] $ 2K_2 \cup K_1 $ 
\item[($b_{4}$)] $P_4$
\end{multicols}
\end{small}
\end{itemize}
\item[($c$)] 
$\chi_D(G\cup G)=\chi_D(G)+2$
if and only if 
$G\cong K_2 \cup \overline{K_2}$.
\item[($d$)] 
$\chi_D(G\cup G)=2\chi_D(G)$
if and only if
$G\cong \overline{K_r} \cup K_2 $, for $ r \geq 2$.
\end{itemize} 
\end{theorem}
\begin{proof}
The graph 
$G$ 
is the join of a complete multipartite graph (possibly vacuous) with one of the graphs presented in Theorem \ref{e1}. 
If 
$G$ 
is the join of a nonvacuous complete multipartite graph with one of the graphs (a) - (q) presented in Theorem \ref{e1}, 
then by Lemma \ref{lem6}, we have the results in 
$(a_1)$.  
Now, let 
$G$
be one of the graphs (a) - (q) that are presented in Theorem \ref{e1}.
In some cases with $\chi_D(G\cup G)=\chi_D(G)$, we will select distinct vertices 
$v_1, v_2, v_3, v_4$ 
of each component of
$G \cup G$ 
for a coloring of 
$G \cup G$ 
such that all vertices other than 
$v_1, v_2, v_3, v_4$ 
receive distinct colors in 
$G$.
Assign to
$v_1$ 
and
$v_2$ 
one color and to
$v_3$ 
and
$v_4$ 
another color in each component of 
$G \cup G$. 
\begin{itemize}
\item[($a_{21}$), ($a_{22}$)] In these cases, there is a singleton color class in each distinguishing coloring. 
Asign color $1$ and color $2$ to the singleton color classes in each component.
\item[($a_{23}$)] Assign color $1$ to $v_1$ and $v_2$ and color $2$ to $v_3$ and $v_4$ in a component. 
Also, assign color $3$ to $v_1$ and $v_2$ and color $2$ to $v_3$ and $v_4$ in the other component.
\item[($a_{24}$)] In this case, we have four triangles $K_3$. The vertices  $v_1$ and $v_3$ are in a triangle in $G$ and the vertices  $v_2$ and $v_4$ are in the other triangle. Assign color $1$ to $v_1$ and $v_2$, color $2$ to $v_3$ and $v_4$ and renaming vertices is colored by colors $3$ and $4$ in a copy of 
$G$. For the other copy,  assign color $3$ to $v_1$ and $v_2$, color $4$ to $v_3$ and $v_4$ and renaming vertices is colored by colors $1$ and $2$. 
\item[($a_{25}$)] In $2K_2 \vee 2K_2$, let the vertices $v_1$ and $v_2$ be in a copy $2K_2$ in $2K_2 \vee 2K_2$ 
and the vertices  $v_3$ and $v_4$ be in the other copy.  Assign color $1$ to $v_1$ and $v_2$ and color $2$ to $v_3$ and $v_4$ in a component of 
$G \cup G$. For the other component,  assign color $3$ to $v_1$ and $v_2$ and color $4$ to $v_3$ and $v_4$. 
\item[($a_{26}$)]  Since $r\geq 2$, $\hat{K}_{2}(\overline{K_r}, H_1) \ncong K_{n-2}, \overline{K_{n-2}}$ 
and 
$|\hat{K}_{2}(\overline{K_r}, H_1)|\geq 3$. 
Let $v_1 \in H_1, v_2, v_3 \in K_2$ and $v_4 \in \overline{K_r}$ in a copy of 
$G$. Assign color $1$ to $v_1$ and $v_2$, color $2$ to $v_3$ and $v_4$ and color $3$ to an other vertex of $\overline{K_r}$.
For the other copy, let $v_2, v_3 \in K_2$ and $v_1, v_4 \in \overline{K_r}$. 
Assign color $1$ to $v_1$ and $v_2$ and color $3$ to $v_3$ and $v_4$.
\item[($a_{27}$), ($a_{28}$)] In both copies of $G \cup G$, let $v_1, v_2 \in 2K_2$, $v_3 \in K_1$ and 
$v_4 \in \hat{K}_{2}(H_1, H_2)(v_4 \in H_1)$. 
Assign color $1$ to $v_1$ and $v_2$ and color $2$ to $v_3$ and $v_4$ in a copy. 
For the other copy,  assign color $2$ to $v_1$ and $v_2$ and color $1$ to $v_3$ and $v_4$. 
\item[($a_{29}$)]  
Let $v_1 \in K_1=\{v_5\}, v_3 \in K_1=\{v_6\}, v_4 \in H_4$ and $v_2 \in H_1$ in the both components of 
$G \cup G$. 
Assign color $1$ to $v_1$ and $v_2$ and color $2$ to $v_3$ and $v_4$, in a component.
For the other component, assign color $3$ to $v_1$ and $v_2$ and color $4$ to $v_3$ and $v_4$. 
\item[($a_{210}$)]  
In both copies of $G \cup G$, let $v_1 \in K_1=\{v_4\}, v_3 \in K_1=\{v_5\}, v_2 \in H_3$ and $v_4 \in H_2$.
Assign color $1$ to $v_1$ and $v_2$ and color $2$ to $v_3$ and $v_4$, in a component.
For the other component, assign color $1$ to $v_1$ and $v_2$ and color $3$ to $v_3$ and $v_4$.
\item[($a_{211}$)]  
Let $v_1 \in K_1=\{v_1\}, v_3 \in K_1=\{v_5\}, v_2 \in H_2$ and $v_4 \in H_1$, in the both components of 
$G \cup G$. 
Assign color $1$ to $v_1$ and $v_2$ and color $2$ to $v_3$ and $v_4$, in a component.
For the other component, assign color $2$ to $v_1$ and $v_2$ and color $1$ to $v_3$ and $v_4$.
\item[($a_{212}), (b_{1}$)]  
Let 
$G=\hat{K}_{3}(H_1, H_2, H_3) \cup K_2$. 
If $\hat{K}_{3}(H_1, H_2, H_3)$ is a complete graph, then $\hat{K}_{3}(H_1, H_2, H_3) \cup \hat{K}_{3}(H_1, H_2, H_3)$ is not distinguished by 
$\chi_D(G)$ colors. In this case, one can check that $\chi_D(G \cup G)=\chi_D(G)+1$. 
If $\hat{K}_{3}(H_1, H_2, H_3)$ is not a complete graph, by Theorem \ref{5.8}, the result in $(a_{212})$ is immediate.
\item[($a_{213}), (b_{2}), (c)$] 
Let $G=\hat{K}_{2}(H_1, H_2) \cup \overline{K_2}$. If $|\hat{K}_{2}(H_1, H_2)|\leq 3$, then for distinguishing the vertices 
$\overline{K_4}$ 
in 
$G \cup G$, 
we need at least 
$4$ 
colors. 
For result in 
$(c)$, 
let $|\hat{K}_{2}(H_1, H_2)|= 2$. Then 
$\hat{K}_{2}(H_1, H_2)$ 
is isomorphic with 
$K_2$ 
and
$\chi_D(G)=2$.
If $|\hat{K}_{2}(H_1, H_2)|= 3$, then 
$\hat{K}_{2}(H_1, H_2)$ 
is isomorphic with 
$K_3$ 
or 
$P_3$. 
In both cases, $\chi_D(G \cup G)=\chi_D(G)+1$. 
If 
$|\hat{K}_{2}(H_1, H_2)|\geq 4$ 
and $\hat{K}_{2}(H_1, H_2)$ is a complete graph by Theorem \ref{5.8}, then $\chi_D(G \cup G)=\chi_D(G)+1$ and the result in $(b_{2})$ is obtained. Otherwise, we have the results in 
$(a_{213})$.

\item[($a_{214}), (b_{4}$)] 
In this case, if $H_1 = H_2 = K_1$, then $\hat{G}_{5}(H_1, H_2, K_1, K_1)=P_4$ and clearly 
$\chi_D(G \cup G)=\chi_D(G)+1$. This is the result in $(b_{4})$. If $H_1 \neq K_1$ or $H_2 \neq K_1$, 
let $v_1 \in K_1=\{v_3\}, v_3 \in K_1=\{v_4\}, v_2 \in H_1$ and $v_4 \in H_2$ 
in both components of 
$G \cup G$. 
Assign color $1$ to $v_1$ and $v_2$ and color $2$ to $v_3$ and $v_4$ in a component.
For the other component, assign color $2$ to $v_1$ and $v_2$ and color $1$ to $v_3$ and $v_4$.

\item[($a_{215}), (b_{4}$)] 
Similarly to the case ($a_{214}), (b_{4}$), if $H_1 = H_2 = K_1$, then we have the result in $(b_{4})$. So, let 
$H_1 \neq K_1$ or $H_2 \neq K_1$. 
For both components of 
$G \cup G$, let  $v_1 \in K_1=\{v_1\}, v_3 \in K_1=\{v_4\}, v_2 \in H_2$ and $v_4 \in H_1$. 
Assign color $1$ to $v_1$ and $v_2$ and color $2$ to $v_3$ and $v_4$ in a component.
For the other component, assign color $2$ to $v_1$ and $v_2$ and color $3$ to $v_3$ and $v_4$. 

\item[($b_{3}$)]  
In this case, there are four copies of 
$K_2$ 
in 
$G \cup G$. 
For distinguishing those copies, we need 
four colors. Also the copies 
$K_1$ 
are colored by the colors used in 
$K_2$. 
Hence, $\chi_D(G\cup G)= n-1 =\chi_D(G)+1$.
\item[($d$)] 
Since 
$G\cup G=\overline{K_{2r}} \cup 2K_2$ and $r \geq 2$, 
we can color the vertices of $2K_2$ by the colors used in 
$\overline{K_{2r}}$. 
So, 
$\chi_D(G \cup G)=\chi_D(\overline{K_{2r}})=2r$.
\end{itemize}
\end{proof}

\section{Conclusions and Future Research} 
Studying the distinguishing coloring of disconnected graphs depends on the number of colorings of its components. 
Especially for calculating the distinguishing chromatic number of 
$G\cup G$, 
it is necessary to pay attention to unique colorability of 
$G$, 
when 
$G$ 
is a connected graph. 
In this paper, our main focus is on the study of uniquely distinguishing colorable graphs, and with that, we have calculated and characterized the distinguishing chromatic number of the certain disconnected graphs. 
This article can be a starting point for studying the distinguishing chromatic number of disconnected graphs.
We end the paper with the following problems. 
\begin{problem}
Find the family of connected graphs that are uniquely distinguishing colorable or not.
\end{problem} 
\begin{problem}
For what subset of type 2 graphs, is the converse of Theorem \ref{01} correct?
\end{problem}
\begin{problem}
For connected graphs
$G$, 
characterize uniquely distinguishing colorable graphs
$G\cup G$.
\end{problem} 
\begin{problem}
Characterize graphs of type 1.
\end{problem}
\begin{problem}
Characterize the class  $\xi$  of graphs such that 
$G \in \xi$ 
if and only if 
$\chi_{D}(G\cup G) = \chi_{D}(G)$.
\end{problem}

%To conclude the paper, we state some open problems here.

%\noindent {\bf Problem A.} 

%\noindent{\bf Acknowledgment.}
 %The authors are deeply grateful to Professors R. Kalinowski and   M. Pil$\acute{\textrm{s}}$niak for participant in a workshop about `Distinguishing Colorable Graphs' in  Ferdowsi university of Mashhad in 2018. 

\end{document}